\documentclass[a4paper,12pt,reqno]{amsart}

 \usepackage[T2A]{fontenc}
 \usepackage[utf8]{inputenc}
 \usepackage[ukrainian,english]{babel}
 \usepackage{amssymb,upref}
\usepackage{amsthm}

\usepackage{graphicx}
\usepackage{mathrsfs}

\usepackage[text={130mm,200mm}]{geometry}

\theoremstyle{plain}
\newtheorem{theorem}{Theorem}
\newtheorem*{theorem*}{Theorem}
\newtheorem{corollary}{Corollary}
\newtheorem*{corollary*}{Corollary}
\newtheorem{lemma}{Lemma}
\newtheorem*{lemma*}{Lemma}

\newtheorem*{proposition*}{Proposition}

\newtheorem*{conjecture*}{Conjecture}
\theoremstyle{definition}
\newtheorem{definition}{Definition}
\newtheorem*{definition*}{Definition}
\theoremstyle{remark}
\newtheorem{remark}{Remark}
\newtheorem*{remark*}{Remark}

\begin{document}

\title[Some applications of binary numeration systems with nonzero redundancy]{Some applications of binary numeration systems with nonzero redundancy to the theory of locally complex functions}

\author[S. Vaskevych]{Svitlana Vaskevych}
\email{svetaklymchuk@imath.kiev.ua}
\address{Institute of Mathematics of NAS of Ukraine, ORCID: 0000-0009-0005-3979}

\author[Yu. Vovk]{Yuliia Vovk}
\email{freeeidea@ukr.net}
\address{K. D. Ushynskyi Chernihiv Regional Institute of Postgraduate Pedagogical Education, Chernihiv, Ukraine}

\author[O. Pratsiovytyi]{Oleksandr Pratsiovytyi}
\email{o.m.pratsovytyi@udu.edu.ua}
\address{Ukrainian State Dragomanov University, Kyiv, Ukraine}

\thanks{This work was supported by a grant from the Simons Foundation (SFI-PD-Ukraine-00014586 V.S.)}


\begin{abstract}
In this paper, we study a numeration system with a non--integer base $a>1$ over the binary alphabet $A=\{0,1\}$:
$$\left[0;\frac{r}{a-1}\right]\ni
x=\sum\limits_{n=1}^{\infty}\frac{\alpha_n}{a^n}
\equiv
\Delta^{r_a}_{\alpha_1\alpha_2\ldots\alpha_n\ldots},
\quad
\alpha_n\in A.$$
We investigate the geometry of $r_a$--representation of numbers, including the geometric interpretation of digits, the structure of cylinder overlaps, and the associated metric properties.

The main object of our study is a nowhere monotone function of unbounded variation defined by
$$f\left(x=\sum\limits_{n=1}^{\infty}\frac{\alpha_n}{2^n}\right)
=\sum\limits_{n=1}^{\infty}\frac{\alpha_n}{a^n},$$
where the classical binary representation of the argument is assumed not to end with the infinite period $(1)$. We derive a system of two functional equations satisfied by $f$. We prove that the function has fractal level sets and a self--affine graph. We also evaluate the integral of $f$ over the interval $[0,1]$.
\end{abstract}

\maketitle

\section{Introduction}
We refer to a function as \emph{locally complex} if every arbitrarily small interval of its domain contains nonempty sets of features of various types (e.g., variational, integro-differential, etc.). Traditionally, this class includes continuous functions that are nowhere monotone, non differentiable, singular, as well as functions having a countable dense set of discontinuities while remaining continuous at all other points. The present paper is devoted to the latter class of functions.

Let $a>1$ be a fixed real number, let $r$ be a positive integer satisfying $r\geq a-1$, let $A=\{0,1,\ldots,r\}$ be the alphabet, and let $L=A\times A\times\cdots$ denote the set of all infinite sequences over the alphabet $A$.

Consider the mapping $\gamma:L\to\left[0;\frac{r}{a-1}\right]$ defined by
\[
\gamma((\alpha_n))
=
\sum\limits_{n=1}^{\infty}\frac{\alpha_n}{a^n}
\equiv
\Delta^{r_a}_{\alpha_1\alpha_2\ldots\alpha_n\ldots}
=x
\in
\left[0;\frac{r}{a-1}\right],
\quad
(\alpha_n)\in L.
\]
The symbolic expression $\Delta^{r_a}_{\alpha_1\alpha_2\ldots\alpha_n\ldots}$
is called the \emph{$r_a$--representation} of the number $x$, while $\alpha_n$ is referred to as the \emph{$n$th digit} of this representation.

The binary alphabet case, i.e., $r=1$, deserves special attention because of the minimality of the alphabet $A={0,1}$ and the technical convenience of binary numeral systems \cite{pr3}. In this case, every number in the interval $\left[0;\frac{r}{a-1}\right]$ can be represented as the sum of a subseries of the geometric series $\frac{1}{a}+\frac{1}{a^2}+\cdots+\frac{1}{a^n}+\cdots$, is the attractor of the iterated function system generated by the contractions $\gamma_i(x)=\frac{1}{a}x+\frac{i}{a}$, $i=0,1$.

If $a=r+1$, then the $r_a$-representation coincides with the classical positional representation of numbers in the integer base $a$. This representation has zero redundancy: every number has at most two representations, and the set of numbers admitting two distinct representations is countable. This case has been extensively studied. For such numeral systems, the topological, metric, fractal, and probabilistic aspects of the theory are well developed \cite{pr4}, and numerous applications have been established. The theory for non-integer bases has been developing since the pioneering work of Rényi \cite{1} in 1957 and was further developed by Parry \cite{2}. Since then, it has experienced rapid progress \cite{erd, 14_Komornik, 22_komornik, komornik_cant}.

\begin{definition}
Assume that $a$ is rational. A number $x$ is called \emph{$r_a$--rational} if it admits an $r_a$--representation whose period is $(0)$.
\end{definition}

However, not every rational number is $r_a$--rational. For example, let $x$ be a number with the purely periodic $r_a$--representation $\Delta^{r_a}_{(c_1\ldots c_p)}$, whose period contains at least two distinct digits. Then $x$ is rational but not $r_a$--rational. Indeed,
\begin{align*}
x = & \Delta^{r_a}_{(c_1\ldots c_p)}
= \left(\frac{c_1}{a}+\cdots+\frac{c_p}{a^p}\right)+ \frac{1}{a^p}\left(\frac{c_1}{a}+\cdots+\frac{c_p}{a^p}\right)+ \frac{1}{a^{2p}}\left(\frac{c_1}{a}+\cdots+\frac{c_p}{a^p}\right) +\cdots\\ 
= & \frac{c}{1-\frac{1}{a^p}},
\end{align*}
where $c=\frac{c_1}{a}+\cdots+\frac{c_p}{a^p}$ is a rational number but, in general, is not $r_a$-rational.

For integer bases $a$ with the standard alphabet $(r=a-1)$, the criterion for the rationality of a number is well known: a number $x\in[0,1]$ is rational if and only if its representation in the classical base-$a$ numeral system is eventually periodic. For non-integer bases, however, no analogous criterion is known.

\section{Geometry of the $r_a$--representation}

Each pair of parameters $a$ and $r$ gives rise to its own geometry. This geometry is partially characterized by the properties of cylinder sets. The following cases deserve special attention: $a$ is an integer, $a$ is a rational, and $a$ is an irrational number.
The first case was studied in papers \cite{1_dist_rv,pr_rat_ms_26,prats_vynnysh,pratsmykyt}.

\begin{definition}
Let $(c_1,c_2,\ldots,c_k)$ be a fixed ordered sequence of elements of the alphabet.
The \emph{cylinder} (or \emph{$r_a$--cylinder}) of rank $k$ with base $c_1c_2\ldots c_k$ is the set
$$\Delta^{r_a}_{c_1c_2\ldots c_k} = \left\{
x: x=\Delta^{r_a}_{c_1\ldots c_k\alpha_1\alpha_2\ldots\alpha_n\ldots},
(\alpha_n)\in L \right\}. $$
\end{definition}

The cylinders $\Delta^{r_a}_{c_1\ldots c_{m-1}j}$ and $\Delta^{r_a}_{c_1\ldots c_{m-1}[j+1]}$, where $j\in\{0,1,\ldots,r-1\}$, are called adjacent. They are subcylinders of the same parent cylinder $\Delta^{r_a}_{c_1\ldots c_{m-1}}$ of the preceding rank.

The cylinder sets satisfy the following properties:

1) $\Delta^{r_a}_{c_1\ldots c_k} =\Delta^{r_a}_{c_1\ldots c_k0} \cup \Delta^{r_a}_{c_1\ldots c_k1} \cup \cdots \cup \Delta^{r_a}_{c_1\ldots c_kr}$;

2) $\Delta^{r_a}_{c_1\ldots c_k}=[u;u+d]$, where $u=\sum\limits_{i=1}^{k}\frac{c_i}{a^i}$, $d=\frac{r}{a^{k}(a-1)}$;

3) the length of the cylinder is
$|\Delta^{r_a}_{c_1\ldots c_k}| =d=\frac{r}{a^k(a-1)} \to 0 \quad (k\to\infty)$;

4) $ \bigcap\limits_{k=1}^{\infty} \Delta^{r_a}_{c_1\ldots c_k} = \Delta^{r_a}_{c_1c_2\ldots c_k\ldots} =x$ for every sequence $(c_k)\in L$;

5) $\min \Delta_{c_1c_2\ldots c_kc} < \min \Delta_{c_1c_2\ldots c_k[c+1]}, \quad 0\leqslant c\leqslant r-1$;

6) $\Delta^{r_a}_{\alpha_1\ldots\alpha_k} = \Delta^{r_a}_{\beta_1\ldots\beta_k}
\iff \sum\limits_{i=1}^{k} a^{-i}(\alpha_i-\beta_i)=0$;

7) $\Delta^{r_a}_{c_1\ldots c_kc} \cap \Delta^{r_a}_{c_1\ldots c_k[c+1]} = \left[\Delta^{r_a}_{c_1\ldots c_k[c+1](0)}; \Delta^{r_a}_{c_1\ldots c_kc(r)}\right] \neq \varnothing$; 

8) the length of the overlap is
$$
\left|\Delta^{r_a}_{c_1\ldots c_{k-1}c}
\cap
\Delta^{r_a}_{c_1\ldots c_{k-1}[c+1]}\right|
=
\Delta^{r_a}_{c_1\ldots c_{k-1}c(r)}
-
\Delta^{r_a}_{c_1\ldots c_{k-1}[c+1](0)}
=
\frac{r-a+1}{a^k(a-1)};
$$

9) the condition
$$
\Delta^{r_a}_{c_1\ldots c_{k-1}c}
\cap
\Delta^{r_a}_{c_1\ldots c_{k-1}[c+1]}
=
\Delta^{r_a}_{c_1\ldots c_{k-1}cr}
=
\Delta^{r_a}_{c_1\ldots c_{k-1}[c+1]0}
$$
is equivalent to
$$
\frac{c}{a^k} + \frac{r}{a^{k+1}} = \frac{c+1}{a^k}, \qquad \text{i.e., } r=a,
$$
which is possible only when $a$ is an integer;

10) the equality
$\left|\Delta^{r_a}_{c_1\ldots c_kc_{k+1}}\right| = \frac{1}{2} \left|\Delta^{r_a}_{c_1\ldots c_k}\right|$ holds only if $a=2$;

11) the condition
$$\Delta^{r_a}_{c_1\ldots c_{k-1}c} \cap \Delta^{r_a}_{c_1\ldots c_{k-1}[c+1]} =
\Delta^{r_a}_{c_1\ldots c_{k-1}c\underbrace{r\ldots r}_m} = \Delta^{r_a}_{c_1\ldots c_{k-1}[c+1]\underbrace{0\ldots 0}_m}$$ is equivalent to
$$
\frac{c}{a^k} + \frac{r}{a^{k+1}} +\cdots+ \frac{r}{a^{k+m}} = \frac{c+1}{a^k},
\qquad \text{i.e., } r=\frac{a^m(a-1)}{a^m-1}.
$$

\begin{remark}
Properties 8)--11) describe the specific features of the overlaps of $r_a$--cylinders.
\end{remark}

\begin{lemma}
If $r=1$, then $\Delta^{r_a}_{0\underbrace{1\ldots 1}_k}\subset \Delta^{r_a}_{0}\cap \Delta^{r_a}_{1}$ for all $k\geqslant k_0$, where $k_0$ is the smallest solution of the inequality
\begin{equation}\label{(*)} 
a^{k-1}+a^{k-2}+\cdots+a^0\ge a^k. 
\end{equation}
\end{lemma}

\begin{proof}
Indeed, $\Delta^{r_a}_{0\underbrace{1\ldots 1}_k}=\left[\Delta^{r_a}_{0\underbrace{1\ldots 1}_k(0)}; \frac{1}{a(a-1)}\right]$, $\Delta^{r_a}_{1}=\left[\frac{1}{a}; \frac{1}{a-1}\right]$.

Hence, $\Delta^{r_a}_{0\underbrace{1\ldots 1}_k}\subset\Delta^{r_a}_{1}$ whenever
$\Delta^{r_a}_{0\underbrace{1\ldots 1}_k(0)}\geqslant\frac{1}{a}$, that is,
$\frac{1}{a^2}+\frac{1}{a^3}+\ldots+\frac{1}{a^{k+1}}\geqslant\frac{1}{a}$,
which is equivalent to \eqref{(*)}.
\end{proof}

\begin{lemma}
If $r=1$ and $a=\frac{1+\sqrt5}{2}$, then $\Delta^{r_a}_{0}\cap \Delta^{r_a}_{1}
= \Delta^{r_a}_{011} = \Delta^{r_a}_{100}$.
\end{lemma}

\begin{proof}
Indeed, $\Delta^{r_a}_{0}\cap \Delta^{r_a}_{1} = \left[ \Delta^{r_a}_{1(0)}; \Delta^{r_a}_{0(1)} \right]  = \left[\frac{1}{a}; \frac{1}{a(a-1)}\right] =
\Delta^{r_a}_{100} = \Delta^{r_a}_{011}$, since 
$a=a^2-1$.
\end{proof}

Consider the set $C$ of all numbers $x\in\left[0;\frac1{a-1}\right]$ whose $r_a$--representation contains no block of fewer than $m$ consecutive zeros. That is, $x=\Delta^{r_a}_{\alpha_1\alpha_2\ldots\alpha_n\ldots}\in C $ if and only if $\alpha_k=0$ implies that $\alpha_{k+j}=0$ for every $j=\overline{1,\ldots,m-1}$.

\begin{theorem}
If $r=1$ and $m$ is the smallest positive integer satisfying the inequality
$a^{m-1}(a-1)>1$, then the set $C$ is a self-similar Cantor-type set of Lebesgue measure zero. Its Hausdorff--Besicovitch dimension is the unique solution of the equation
\[
\frac{1}{a^x}+\frac{1}{a^{mx}}=1.
\]
\end{theorem}

\begin{proof}
Since
\begin{align*}
  \min \Delta^{r_a}_1-\max \Delta^{r_a}_{\underbrace{0\ldots0}_m}= & \Delta^{r_a}_{1(0)}
- \Delta^{r_a}_{\underbrace{0\ldots0}_m(1)}=\frac{1}{a} - \left(\frac{1}{a^2}+\cdots+\frac{1}{a^{m+1}}\right) \\
  = & \frac1a-\frac1{a^m(a-1)}>0
\end{align*}
it follows that
\[
C=[\Delta^{r_a}_{\underbrace{0\ldots0}_m}\cap C] \cup [\Delta^{r_a}_1\cap C].
\]
Moreover, $C$ is similar to $\Delta^{r_a}_{\underbrace{0\ldots0}_m}\cap C$
with similarity ratio $k_1=\frac{1}{a^m}$, and $C$ is similar to $\Delta^{r_a}_1\cap C$
with similarity ratio $k_2=\frac{1}{a}$.

Hence, the Lebesgue measure of the set $C$ satisfies
\[
\lambda(C)=k_1\lambda(C)+k_2\lambda(C)=(k_1+k_2)\lambda(C).
\]
Since $k_1+k_2\neq1$, we conclude that $\lambda(C)=0$.

The Cantor--type set $C$ is self--similar, and its similarity dimension is the unique solution of the equation
\[
\frac{1}{a^x}+\frac{1}{a^{mx}}=1.
\]
Since $C$ satisfies the open set condition \cite{triebel}, its similarity dimension coincides with its Hausdorff--Besicovitch dimension.
\end{proof}

\begin{lemma}\label{eq_a}
If $r=1$ and $\Delta^{r_a}_{0}\cap \Delta^{r_a}_{1} = \Delta^{r_a}_{0\underbrace{1\ldots 1}_m}=\Delta^{r_a}_{1\underbrace{0\ldots 0}_m}$,
where $m\geqslant 2$, then $a$ is the positive root of the equation
$a^{m+1}-2a^{m}+1=0$.
\end{lemma}

\begin{proof}
For $r=1$, Property~11 yields $1=\frac{a^m(a-1)}{a^m-1}$, which is equivalent to
$a^{m+1}-2a^{m}+1=0$.
\end{proof}

\begin{corollary}
If $r=1$ and $a=\frac{1+\sqrt5}{2}$, then the blocks of three consecutive digits $100$ and $011$ are interchangeable (equivalent) in the $r_a$--representations of numbers from the interval $\left[0;\frac{r}{a-1}\right]$.
\end{corollary}

\begin{theorem}
If $r=1$ and $a=\frac{1+\sqrt5}{2}$, then every number $x\in\left(0;\frac{r}{a-1}\right)$
has infinitely many distinct $r_a$--representations. Moreover, if one of the $r_a$--representations of $x$ contains infinitely many zeros and ones does not contain the period $(01)$, then $x$ has continuum many distinct $r_a$--representations. 
\end{theorem}

\begin{proof}
1. A number whose representation has the simple period $(i)$ has at least countably many distinct $r_a$-representations. Indeed,
\begin{align*}
  \Delta^{r_a}_{c_1\ldots c_{m-1}1(0)} = & \Delta^{r_a}_{c_1\ldots c_{m-1}011(0)} = \Delta^{r_a}_{c_1\ldots c_{m-1}01011(0)} = \ldots = \Delta^{r_a}_{c_1\ldots c_{m-1}\underbrace{0101\ldots01}_n1(0)}  \\
  = & \Delta^{r_a}_{c_1\ldots c_{m-1}(01)}; \\
  \Delta^{r_a}_{c_1\ldots c_{m-1}0(1)} = & \Delta^{r_a}_{c_1\ldots c_{m-1}100(1)}
=
\Delta^{r_a}_{c_1\ldots c_{m-1}10100(1)}
=
\ldots
=
\Delta^{r_a}_{c_1\ldots c_{m-1}\underbrace{1010\ldots10}_n0(1)}
 \\
  = & \Delta^{r_a}_{c_1\ldots c_{m-1}(10)}.
\end{align*}

2. By the previous corollary, the numbers
$x_1=\Delta^{r_a}_{(100)}$ and $x_2=\Delta^{r_a}_{(011)}$ have continuum many distinct representations, since each of their representations contains infinitely many positions at which two alternative blocks may be substituted. For the same reason, any number whose $r_a$--representation contains the block $100$ infinitely many times also has continuum many distinct representations. This occurs whenever the representation contains infinitely many zeros and ones and infinitely many blocks of consecutive zeros of length greater than one. An analogous conclusion holds for the block $011$.

3. It remains to consider the number $x_0=\Delta^{r_a}_{(01)}$. Since $\Delta^{r_a}_{(01)} = \frac{1}{a^2-1} = \frac{1}{a} = \Delta^{r_a}_{1(0)}$,  it follows from Part~1 that $x_0$ has at least countably many distinct representations. The same conclusion holds for every number of the form $x=\Delta^{r_a}_{c_1c_2\ldots c_{m-1}(01)}$.
\end{proof}

\section{A function with fractal level sets}

Consider the function $\varphi$ defined by
\begin{equation}\label{phi_func}
  \varphi(\Delta^2_{\alpha_1\alpha_2\ldots\alpha_n\ldots})=\Delta^{r_a}_{\alpha_1\alpha_2\ldots\alpha_n\ldots}.
\end{equation}

In paper \cite{VasVovkPrats}, it was proved that the function $\varphi$ is continuous at every $2$--adic--unary point and discontinuous at every $2$--adic--binary point. Moreover, the sum of its jump discontinuities is infinite. The function is nowhere monotone and has unbounded variation.

We note that the function $\varphi$ is a particular case of the function
\[
b_{p,q}(x):\;
x=\sum\limits_{i=1}^{\infty}\frac{\alpha_i}{p^i}
\mapsto
\sum\limits_{i=1}^{\infty}\frac{\alpha_i}{q^i}=y,
\]
where $p>1$ and $q>1$, which was called a \emph{Cantor-type function} by the authors of \cite{komornik_cant}. This terminology is nonstandard, since the term \emph{Cantor-type function} usually refers to a continuous nondecreasing singular function whose set of points of increase has Lebesgue measure zero.

We also note that the definition of the function $b_{p,q}(x)$ is well posed because the argument admits a unique representation, obtained by the greedy expansion (greedy algorithm) \cite{1}.

In this paper, we establish several additional properties of the function $\varphi$.

\begin{theorem}
The function $\varphi$ is a solution of the system of functional equations
\begin{equation}
  \begin{cases}\label{system}
  \varphi\!\left(\frac{x}{2}\right)=\frac{1}{a}\varphi(x),\\
  \varphi\!\left(\frac{1}{2}+\frac{x}{2}\right)=\frac{1}{a}+\frac{1}{a}\varphi(x).
\end{cases}
\end{equation}
\end{theorem}

\begin{proof}
We show that, in the class of bounded functions defined on the interval $[0,1]$, the system of functional equations~(\ref{system}) has the function $\varphi$, defined by~(\ref{phi_func}), as its solution.

If a number $t$ is represented by its classical binary expansion,
\[
t=\frac{a_1}{2}+\frac{a_2}{2^2}+\cdots+\frac{a_n}{2^n}+\cdots
=\Delta^{2}_{a_1a_2\ldots a_n\ldots},
\qquad
a_n\in\{0,1\},
\]
then
\[
\frac{t}{2}
=
\Delta^{2}_{0a_1a_2\ldots a_n\ldots},
\qquad
\frac{1+t}{2}
=
\Delta^{2}_{1a_1a_2\ldots a_n\ldots}.
\]

Since the function $\varphi$ is defined at every point of the interval $[0,1]$ and takes only nonnegative values, it follows from~\eqref{system} that
\begin{align*}
  \varphi(x=\Delta^{2}_{\alpha_1\alpha_2\ldots\alpha_n\ldots})
  &=
  \frac{\alpha_1}{a}
  +\frac{1}{a}\varphi(\Delta^{2}_{\alpha_2\alpha_3\ldots\alpha_n\ldots})\\
  &=
  \frac{\alpha_1}{a}
  +\frac{1}{a}
  \left(
  \frac{\alpha_2}{a}
  +\frac{1}{a}\varphi(\Delta^{2}_{\alpha_3\alpha_4\ldots\alpha_n\ldots})
  \right)
  =\cdots\\
  &=
  \frac{\alpha_1}{a}
  +\frac{\alpha_2}{a^2}
  +\cdots
  +\frac{\alpha_k}{a^k}
  +\frac{1}{a^k}\varphi(\Delta^{2}_{\alpha_k\alpha_{k+1}\ldots}).
\end{align*}

Since $\varphi$ is bounded and defined at every point
$x=\Delta^{2}_{\alpha_1\alpha_2\ldots\alpha_k\ldots}$, in particular, at every point of the sequence 
\[
x_k=\Delta^{2}_{\alpha_k\alpha_{k+1}\ldots},
\quad 
k\in N,
\]
we have
\[
\frac{1}{a^k}
\varphi(\Delta^{2}_{\alpha_k\alpha_{k+1}\ldots})
\longrightarrow 0
\quad \text{as} \quad
k\to\infty.
\]

Therefore, the above expansion $\varphi(x=\Delta^{2}_{\alpha_1\alpha_2\ldots\alpha_n\ldots})$ converges, and hence
\[
\varphi(x=\Delta^{2}_{\alpha_1\alpha_2\ldots\alpha_n\ldots})
=
\sum\limits_{n=1}^{\infty}\frac{\alpha_n}{a^n}
=
\Delta^{r_a}_{\alpha_1\alpha_2\ldots\alpha_n\ldots}.
\qedhere
\]
\end{proof}

\begin{lemma}
The graph $\Gamma_{\varphi}$ of the function $\varphi$ is a self--affine subset of $R^2$ whose self--affine dimension is $2\log_{2a}2$.
\end{lemma}

\begin{proof}
Clearly, $\Gamma_{\varphi}=\Gamma_{0}\cup\Gamma_{1}$, where $\Gamma_{i}=\{M(x,y): x\in\Delta^2_{i},\; y=\varphi(x)\}$, $i\in\{0,1\}$, and $\Gamma_i=\gamma_i(\Gamma_{\varphi})$, where $\gamma_i$ is the self--affine transformation defined by
\[
\gamma_i:
\begin{cases}
x'=\dfrac{x}{2}+\dfrac{i}{2}
=\Delta^{2}_{i\alpha_1(x)\alpha_2(x)\ldots\alpha_n(x)\ldots},\\[2mm]
y'=\dfrac{\varphi(x)}{a}+\dfrac{i}{a}
=\Delta^{r_a}_{i\alpha_1(x)\alpha_2(x)\ldots\alpha_n(x)\ldots}.
\end{cases}
\]

Therefore, the self--affine dimension of the graph is the solution of the equation
\[
2\cdot
\left|
\begin{array}{cc}
\dfrac{1}{2} & 0\\
0 & \dfrac{1}{a}
\end{array}
\right|^{\frac{x}{2}}
=1,
\]
that is,
\[
x=2\log_{2a}2.
\]
\end{proof}

\begin{theorem}
The function $\varphi$ has fractal level sets. In particular, if $a$ satisfies the equation $a^{m+1}-2a^m+1=0$, then the level set corresponding to $\Delta^{r_a}_{0\underbrace{1\ldots1}_m}$
is fractal and has Hausdorff--Besicovitch dimension at least $\frac{1}{m+1}$.
\end{theorem}

\begin{proof}
By Lemma~\ref{eq_a}, the condition $a^{m+1}-2a^m+1=0$ implies
$\Delta^{r_a}_{0\underbrace{1\ldots1}_m}=\Delta^{r_a}_{1\underbrace{0\ldots0}_m}$.
The preimages of these cylinder sets are the binary cylinders $\Delta^{2}_{0\underbrace{1\ldots1}_m}$ and $\Delta^{2}_{1\underbrace{0\ldots0}_m}$,
which are disjoint. Each of these cylinders has length
$\frac{1}{2^{m+1}}$.

Every point $x=\Delta^2_{\overline{a}_1\overline{a}_2\ldots\overline{a}_n\ldots}$,
$\overline{a}_n\in \left\{0\underbrace{1\ldots1}_m,\;1\underbrace{0\ldots0}_m\right\}$,
is a preimage of the point $y_0=\Delta^{r_a}_{(0\underbrace{1\ldots1}_m)}$.
The set of all such numbers is, in general, a subset of the level set corresponding to $y_0$. Hence, the Hausdorff--Besicovitch dimension of the level set
$\varphi^{-1}(y_0)$ is at least the Hausdorff--Besicovitch dimension of the set of numbers $x\in[0,1]$ whose representations in the numeration system with base $2^{m+1}$ use only two digits. It is well known that the latter dimension is the solution of the equation $2\cdot2^{-(m+1)x}=1$, that is, $x=\frac{1}{m+1}$.
\end{proof}

\begin{theorem}
The following equality holds: $\int\limits_{0}^{1}\varphi(x)\,dx=\frac{1}{2(a-1)}$.
\end{theorem}

\begin{proof}
Using the additivity of the integral together with the functional equations~(\ref{system}) satisfied by the function $\varphi$, we obtain
\begin{align*}
\int\limits_{0}^{1}\varphi(x)\,dx
&=
\int\limits_{0}^{\frac{1}{2}}\varphi(x)\,dx
+
\int\limits_{\frac{1}{2}}^{1}\varphi(x)\,dx \\
&=
\int\limits_{0}^{1}
\frac{1}{a}\varphi(x)\,d\!\left(\frac{x}{2}\right)
+
\int\limits_{0}^{1}
\left[
\frac{1}{a}
+
\frac{1}{a}\varphi(x)
\right]
d\!\left(\frac{1}{2}+\frac{x}{2}\right) \\
&=
\frac{1}{2a}\int\limits_{0}^{1}\varphi(x)\,dx
+
\frac{1}{2a}\int\limits_{0}^{1}\varphi(x)\,dx
+
\frac{1}{2a}\int\limits_{0}^{1}dx \\
&=
\frac{1}{a}\int\limits_{0}^{1}\varphi(x)\,dx
+
\frac{1}{2a}.
\end{align*}

Hence, $\left(1-\frac{1}{a}\right)\int\limits_{0}^{1}\varphi(x)\,dx=\frac{1}{2a}$,
which implies
\[
\int\limits_{0}^{1}\varphi(x)\,dx
=
\frac{1}{2(a-1)}.
\]
\end{proof}

\section{Concluding remarks}

In this paper, we have studied the structural, fractal, and automodel properties of the function $\varphi$, which is a particular case of the class of functions defined by \eqref{phi_func}, corresponding to $r=1$. Extending these results to arbitrary $r\in N$, including the investigation of the topological, metric, and fractal properties of all level sets, appears to be a considerably more challenging problem.

Another promising direction for future research is the study of the distribution of the random variable $Y=\varphi(X)$, where $X$ is either uniformly or exponentially distributed on the interval $[0,1]$. We plan to address these questions in future work.


\end{document}